\documentclass[12pt]{amsart}
\usepackage{amssymb}
\usepackage{}
\usepackage{cases}

\usepackage{amsmath}
\usepackage{txfonts}
\theoremstyle{plain}
\newtheorem{Thm}{Theorem}

\newtheorem{Lem}[Thm]{Lemma}

\begin{document} %begin Topmatter
\title[fractional Laplacian]
{On nonlocal nonlinear elliptic problem with the fractional Laplacian}

\author{Li Ma}

\address{Ma: Department of mathematics \\
Henan Normal university \\
Xinxiang, 453007 \\
China} \email{lma@tsinghua.edu.cn}

\dedicatory{}
\date{May 26th, 2014}

\begin{abstract}

In this paper, we study a nonlocal elliptic  problem with the fractional Laplacian on $R^n$. We show that the problem
has infinite positive solutions in $C^\tau(R^n)\bigcap H^\alpha_{loc}(R^n)$. Moreover each of these solutions tends to some positive constant limit at infinity.


\end{abstract}

\thanks{$^*$ The research is partially supported by the National Natural Science
Foundation of China (No. 11271111) and SRFDP 20090002110019. }
\maketitle

\section{Introduction}

We prove the following result.

\begin{Thm}\label{main} Assume $0<\alpha<n$ and $p>1$. Let $\omega :R_+\to R_+$ be the monotone non-increasing function such that
\begin{equation}\label{cond}
\int_1^{\infty}\frac{\omega(r)}{r}dr=A<\infty¡£
\end{equation}
Then there is a positive constant $\theta$ such that for smooth functions $k(x)$ and $K(x)$ with $|K(x)|\leq \theta \omega(|x|)(1+|x|)^{-\tau}$ and
$0\leq k(x)\leq \theta \omega(|x|)(1+|x|)^{-\tau}$ on $R^n$ for some $\tau\geq \alpha$,
the problem 
\begin{equation}\label{nonlocal}
(-\Delta)^{\alpha/2}u+k(x)u=K(x)u^p,  \ \ \  in \ \ R^n
\end{equation}
has infinite positive solutions in $C^\tau(R^n)\bigcap H^\alpha_{loc}(R^n)$. Moreover, each of these solutions tends to some positive constant limit at infinity.
\end{Thm}

We prove the above result by using the Perron method and similar argument as in Lin's work \cite{L}. The Perron method is based on the maximum principle developed in \cite{CS}. Similar argument had been carried out in \cite{M} for nonlinear subelliptic equations on $R^n$.

We denote by $C$ the uniform constants, which may vary in different inequalities or formulae.

Denote by $B_1(0)$ the unit ball in $R^n$.

We present the potential analysis as in \cite{LN} in section \ref{sect2} and we prove the main result in section \ref{sect3}.

\section{Preliminary}\label{sect2}

Let $\omega:R^n\to R$ be a radially symmetric monotone non-increasing function.
Let $f:R^n\to R$ be a locally Holder continuous function with the decay growth as below:
\begin{equation}\label{f1}
|f(x)|\leq C\omega(|x|)|x|^{-\tau}
\end{equation}
where $C>0$  and $\tau>2\alpha$ are uniform constants.

\begin{Lem}\label{Lem1} Define
$$
w(x)=c_n\int_{R^n} \frac{f(y)}{|x-y|^{n-\alpha}}dy,
$$
where $c_n=[n(n-2)|B_1(0)|]^{-1}$. Then $w(x)$ is well-defined and near $\infty$ we have
\begin{equation}\label{upper}
|w(x)|\leq \left\{
\begin{array}{l l l}
C|x|^{\alpha-n}\omega(|x|) & \quad \mbox{if $\tau>n$}\\
C\omega(|x|)\log |x| & \quad \mbox{if $\tau=n$ }\\
C\omega(|x|)|x|^{\alpha-\tau}  & \quad \mbox{if $\alpha\leq\tau $}\\
\end{array} \right.
\end{equation}
where $c_n=1/n(n-2)|B_1(0)|$.
\end{Lem}

\begin{proof}
Clearly we have
$$
|w(x)|\leq C\int_{R^n} \frac{\omega(|x|)}{|x-y|^{n-\alpha}(1+|y|^\tau}dy.
$$
We now divide the the whole space $R^n$ into three parts:
$$
D_1=\{y\in R^n; |y-x|\leq |x|/2\},
$$
$$
D_2=\{y\in R^n; |x|/2\leq |y-x|\leq 2|x|\},
$$
and
$$
D_3=\{y\in R^n; |y-x|\geq 2|x|\}.
$$
Set, for $j=1,2,3$,
$$
I_j=\int_{D_j}\frac{\omega(|x|)}{|x-y|^{n-\alpha}(1+|y|^\tau}dy.
$$
Then
$$
|w(x)|\leq C(I_1+I_2+I_3).
$$
One can show as in \cite{LN} that $w(x)$ has the upper bounds as in (\ref{upper}).
\end{proof}

We now assume that $\omega(x)$ satisfies (\ref{cond}).

\begin{Lem}\label{Lem2} Assume $f\geq 0$ on $R^n$ and $f(x)\geq C\omega(|x|)|x|^{-\tau}$ for some $\tau\geq\alpha$ and $C>0$. The Riesz potential $w(x)$ defined in Lemma 1 has the following lower bounds at $\infty$:
\begin{equation}\label{lower}
w(x)\geq \left\{
\begin{array}{l l l}
C|x|^{\alpha-n}\omega(|x|) & \quad \mbox{if $\tau>n$}\\
C\omega(|x|)\log |x| & \quad \mbox{if $\tau=n$ }\\
C\omega(|x|)|x|^{\alpha-\tau}  & \quad \mbox{if $\alpha\leq\tau n$}\\
\end{array} \right.
\end{equation}

\end{Lem}

\begin{proof} As in the proof above, we have
$$
w(x)=(\int_{D_1}+\int_{D_2}+\int_{D_3})\frac{c_n f(y)}{|x-y|^{n-\alpha}}dy\geq \int_{D_2}\frac{c_n f(y)}{|x-y|^{n-\alpha}}dy:=J.
$$
One can show as in \cite{LN} that $J$ has the lower bounds as in (\ref{lower}).

\end{proof}

\section{Proof of Theorem \ref{main}}\label{sect3}
The argument presented below is similar to LIn's work \cite{L}. We just do it briefly.

Take some constants $\theta_1>0$ and $0<a<1$. 
We first define the function $U_a(x)$ on $R^n$ by solving the nonlocal equation
$$
(-\Delta)^{\alpha/2}u=-\frac{C\omega(x)}{(1+|x|)^{\tau}}, \ \ on \ \ R^n
$$
for $C\in (0, \theta_1)$ with $0<U_a(x)\leq a$ and 
$U_a(x)\to a$ at infinity. We can verify that $U_a$ is the lower solution to (\ref{nonlocal}).

Then we define the function $U^a(x)$ on $R^n$ by solving the nonlocal equation
$$
(-\Delta)^{\alpha/2}u=\frac{C\omega(x)}{(1+|x|)^{\tau}}, \ \ on \ \ R^n
$$
for the same $C>0$ with $a\leq U^a(x) <1$ and 
$U^a(x)\to a$ at infinity. We can verify that $U^a$ is the upper solution to (\ref{nonlocal}).

Then we can use the Perron method \cite{CS}\cite{L} to get the desired solution $u(x)$ to (\ref{nonlocal}) such that
$U_a(x)\leq u(x)\leq U^a(x)$ on $R^n$.

This then completes the proof of Theorem \ref{main}.

\end{document}